\documentclass[11pt,twoside]{article}
\usepackage{geometry}
\usepackage[colorlinks,
            linkcolor=blue,
            anchorcolor=black,
            citecolor=blue
            ]{hyperref}
\usepackage{color,xcolor}
\usepackage{marvosym}
\usepackage{titlesec}
\usepackage{fancyhdr}
\usepackage{cite}
\usepackage{ifthen}
\usepackage{pifont}
\usepackage{tabularx}
\usepackage{booktabs}
\usepackage{diagbox}
\usepackage{makecell}
\usepackage{mathtools}
\usepackage{dcolumn}
\usepackage{stmaryrd}
\usepackage{setspace}
\usepackage{indentfirst}
\usepackage{amsmath,amssymb,amscd,bbm,amsthm,mathrsfs,dsfont,bm}
\usepackage{cases}
\usepackage{mathtools}
\usepackage{enumitem}
\usepackage[T1]{fontenc}
\numberwithin{table}{section}

\usepackage{multirow}
\usepackage{multicol}
\usepackage{float}
\usepackage{graphicx}
\usepackage{subfigure}

\newtheorem{theorem}{\sc Theorem}[section]

\newtheorem{lemma}{\sc Lemma}[section]

\theoremstyle{definition}

\newtheorem{remark}{\sc Remark}[section]
\newtheorem{example}{\sc Example}[section]

\newtheorem*{proofa}{\it Proof of Theorem \ref{thm-1}}

\numberwithin{equation}{section}

\usepackage[T1]{fontenc}
\newboolean{first}
\setboolean{first}{true}

\title{\vspace{-2cm}\begin{flushleft} 
  \fontsize{15pt}{\baselineskip}\selectfont \textit{Research Article}
\end{flushleft}\vspace{-0.3cm}
\bf \fontsize{18pt}{\baselineskip}\selectfont
\begin{flushleft}
A lower bound for the beta function
\end{flushleft}
}
\author{\hspace{-8.3cm}  Tiehong Zhao$^1$,\quad Miaokun Wang$^{2,\href{mailto:wangmiaokun@zjhu.edu.cn}{\color{black}{\textrm{\Letter}}}}$}

\date{}

\begin{document}
\maketitle
\vspace{-0.5cm}
\footnote{\hspace{-0.35cm}$^{\href{mailto:chuyuming@zjhu.edu.cn}{\color{black}{\textrm{\Letter}}}}$\  Miaokun Wang}   
\footnote{\texttt{wangmiaokun@zjhu.edu.cn}}
\footnote{\vspace{-2mm}}
\footnote{Tiehong Zhao}   
\footnote{\texttt{tiehong.zhao@hznu.edu.cn}} 
\footnote{\hspace{-0.48cm}\rule{4.8cm}{0.1mm}}
\footnote{\hspace{-0.35cm}$^1$\ \ School of mathematics, Hangzhou Normal University, 311121, People's Republic of China}
\footnote{\hspace{-0.35cm}$^2$\ \ Department of mathematics, Huzhou University, Huzhou 313000, People's Republic of China}

\begin{center}
\begin{small}
\begin{minipage}{135mm}
{\bf Abstract.}
We present a new lower bound for Euler's beta function, $B(x,y)$, which states that the inequality 
\begin{equation*}
    B(x,y)>\frac{x+y}{xy}\left(1-\frac{2xy}{x+y+1}\right)
\end{equation*}
holds on $(0,1]\times(0,1]$, which  improves a lower bound obtained by P. Iv\'{a}dy \cite[Theorem, (3.2)]{Ivady2012} in the case of $0<x+y<1$.
\end{minipage}

\bigskip

\begin{minipage}{135mm}
\end{minipage}

\bigskip

\begin{minipage}{135mm}
\end{minipage}
\end{small}
\end{center}

\section{Introduction}\label{sec-1}

The classical gamma function and its logarithmic derivative (or psi function) are defined as
\begin{equation*}
    \Gamma(x)=\int_0^\infty e^{-t}t^{x-1}dt\quad\text{and}\quad \psi(x)=\frac{\Gamma'(x)}{\Gamma(x)}, \qquad x>0.
\end{equation*}
The derivatives of $\psi$, i.e. $\psi',\psi'',\cdots$, are called polygamma functions. The beta function, as a well-known Euler’s integral of the first kind, is defined by
\begin{equation*}
    B(x,y)=\int_0^1 t^{x-1}(1-t)^{y-1}dt, \quad x,y>0,
\end{equation*}
which is closely related to the gamma function by a elegant identity
\begin{equation*}
    B(x,y)=\frac{\Gamma(x)\Gamma(y)}{\Gamma(x+y)}.
\end{equation*}
The utility of the beta function is often overshadowed by that of the gamma function, partly perhaps because it can be evaluated in terms of the gamma function. In fact, it has various applications not only in the theory of special functions, but it also plays a role in other fields, for instance, statistics, mathematical physics, and graph theory; see \cite{AC2009,Dwyer,Rao1981}.

In the recent past years, numerous papers on remarkable inequalities related to the gamma and polygamma functions were published (see \cite{YT2023,Atale2022,CM2022,PM2022,Yinl2022,LZW2022,Qi2021,DS2020} and and the references therein), only few inequalities concerning the beta function can be found in the literature \cite{Alzer2001,Alzer2003,Cerone2007, DAB2000,FR2022}. 

In 2000, S. Dragomir at al. \cite{DAB2000} proved that the inequality 
\begin{equation*}
    0\leq\frac{1}{xy}-B(x,y)\leq\frac{1}{4}
\end{equation*}holds for $x,y\geq1$. Later, Alzer \cite{Alzer1997} improved $1/4$ to the best possible constant $\max\limits_{x\geq1}\{\Delta(x)\}=0.08731\cdots$, where $\Delta(x)=1/x^2-\Gamma^2(x)/\Gamma(2x)$.

For all $x,y\in(0,1]$, sharp rational bounds for the difference $1/(xy)-B(x,y)$ were obtained in \cite{Alzer2003}, which are given by
\begin{equation}\label{intro-1}
    \frac{1}{xy}\left[1-\alpha\frac{(1-x)(1-y)}{(1+x)(1+y)}\right]\leq B(x,y)\leq\frac{1}{xy}\left[1-\beta\frac{(1-x)(1-y)}{(1+x)(1+y)}\right]
\end{equation}with the best possible constants $\alpha=\frac{2}{3}\pi^2-4=2.57973\cdots$ and $\beta=1$. 

In 2012, Iv\'{a}dy \cite{Ivady2012} strengthened the right side of \eqref{intro-1}, although there were some errors in the proof and had been corrected in \cite{Ivady2016}, in which a double inequality for $B(x,y)$ was presented, explicitly, for all $x,y\in(0,1]$, the inequalities
\begin{equation}\label{inequal-Ivady}
    \frac{x+y-xy}{xy}\leq B(x,y)\leq\frac{x+y}{xy(1+xy)}
\end{equation}hold with equality only for $x=y=1$. We remark that not only the upper bound of \eqref{intro-1} has been improved, but the lower bound of \eqref{inequal-Ivady} is also better than that of \eqref{intro-1} for $x,y\in(0,1]$ with $x+y\leq1$.
Indeed, due to $\alpha>5/2$, for $x+y\leq1$,
\begin{align*}
&x+y-xy-\left[1-\alpha\frac{(1-x)(1-y)}{(1+x)(1+y)}\right]>x+y-xy-\left[1-\frac{5}{2}\frac{(1-x)(1-y)}{(1+x)(1+y)}\right]\\
&=\frac{(1-x)(1-y)(3-2x-2y-2xy)}{2(1+x)(1+y)}\geq\frac{(1-x)(1-y)\left[3-2x-2y-(x+y)^2\right]}{2(1+x)(1+y)}\\
&=\frac{(1-x)(1-y)(1-x-y)(3+x+y)}{2(1+x)(1+y)}\geq0.
\end{align*}
The aim of this article is to present a new lower bound of $B(x,y)$ for all $x,y\in(0,1]$.
\begin{theorem}\label{thm-1}
Let $x,y\in(0,1]$. Then the inequality
\begin{equation*}
    B(x,y)>\frac{x+y}{xy}\left(1-\frac{2xy}{x+y+1}\right)
\end{equation*}holds.
\end{theorem}
\begin{remark}
It was observed that  
\begin{equation*}
    1-\frac{xy}{x+y}\geq 1-\frac{2xy}{x+y+1}
\end{equation*}holds for $x+y\geq1$ and $xy>0$, which together with \eqref{inequal-Ivady} implies
\begin{equation}\label{sec3-1}
    B(x,y)\geq  \frac{x+y}{xy}\left(1-\frac{xy}{x+y}\right)\geq \frac{x+y}{xy}\left(1-\frac{2xy}{x+y+1}\right),
\end{equation}where the equality cannot hold simultaneously. This enables us to prove the Theorem \ref{thm-1} only in the case of $0<x+y<1 $. Moreover, if Theorem \ref{thm-1} is true, then 
\begin{equation*}
    B(x,y)>\frac{x+y}{xy}\left(1-\frac{2xy}{x+y+1}\right)\geq\frac{x+y}{xy}\left(1-\frac{xy}{x+y}\right)
\end{equation*}holds for $x,y\in(0,1]$ with $x+y\leq1$, of which in the case, that is to say, our lower bound is better than that of \eqref{inequal-Ivady}. 
\end{remark}

\section{Lemmas}

Before proving the main theorem, we need the following four lemmas.
The first one is to estimate a lower bound for the difference of two psi functions, which is contributed by H Alzer \cite[Theorem 7]{Alzer1997}

\begin{lemma}\label{lm-1}
Let $n\geq0$ be an integer and $s\in(0,1)$. Then the inequality
\begin{equation*}
   \psi(x+1)-\psi(x+s)>(1-s)\left[\frac{1}{x+s+n}+\sum_{i=0}^{n-1}\frac{1}{(x+i+1)(x+i+s)}\right] 
\end{equation*}holds for $x>0$.
\end{lemma}

The following lemma plays an important role in the proof of the main theorem, which gives the rational bounds for the derivatives of psi function. In \cite{Yang2014}, Yang introduced the function 
\begin{equation*}
    \mathcal{L}(x,a)=\frac{1}{90a^2+2}\log\left(x^2+x+\frac{3a+1}{3}\right)+\frac{45a^2}{90a^2+2}\log\left(x^2+x+\frac{15a-1}{45a}\right),
\end{equation*}which will be approximated to psi function. Taking $a=2/5$ or $4/5$, several partial derivatives of $\mathcal{L}(x,a)$ are given by
\begin{small}
  \begin{align*}
 \mathcal{L}_x(x,2/5)&=\frac{3(1+2x)(61+90x+90x^2)}{2(11+15x+15x^2)(5+18x+18x^2)},\\
 \mathcal{L}_x(x,4/5)&=\frac{3(1+2x)(199+180x+180x^2)}{2(17+15x+15x^2)(11+36x+36x^2)},\\
 \mathcal{L}_{xx}(x,2/5)&=-\frac{3(4993+36546x+110526x^2+196560x^3+219780x^4+145800x^5+48600x^6)}{2(11+15x+15x^2)^2(5+18x+18x^2)^2},\\
  \mathcal{L}_{xx}(x,4/5)&=-\frac{3(46537+322206x+784446x^2+1118880x^3+1045440x^4+583200x^5+194400x^6)}{2(17+15x+15x^2)^2(11+36x+36x^2)^2}.
\end{align*} 
\end{small}
In particular, it has been proved in \cite[Lemma 1]{Yang2014} that $a\mapsto\mathcal{L}_x(x,a)$ is decreasing and $a\mapsto\mathcal{L}_{xx}(x,a)$ is increasing on $(1/15,\infty)$. This together with \cite[Corollary 3.3]{ZYC2015} gives the following lemma.

\begin{lemma}
 Let $a_1=\frac{40+3\sqrt{205}}{105}=0.79003\cdots$, $a_2=\frac{45-4\pi^2+3\sqrt{4\pi^4-80\pi^2+405}}{30(\pi^2-9)}=0.47053\cdots$ and $a_3=0.43218\cdots$ be is the unique solution of the equation $\mathcal{L}_{xx}(0,a)=\psi''(1)$.
Then the inequalities
\begin{align}\label{lm2.2-1}
   \mathcal{L}_x(x,4/5)<\mathcal{L}_x(x,a_1)&<\psi'(x+1)<\mathcal{L}_x(x,a_2)<\mathcal{L}_x(x,2/5) \\ \label{lm2.2-2}
   \mathcal{L}_{xx}(x,2/5)<\mathcal{L}_{xx}(x,a_3)&<\psi''(x+1)<\mathcal{L}_{xx}(x,a_1)<\mathcal{L}_{xx}(x,4/5)
\end{align}
hold for all $x>0$ with the best possible constants $a_1,a_2$ and $a_3$. 
\end{lemma}

\begin{lemma}\label{lm-3}
Let $0<x<(\sqrt{3}+1)/2$ and define the function
\begin{equation*}
    f(x)=\log\left[\frac{\Gamma(x+1)^2}{\Gamma(2x+1)}\right]-\log\left(1-\frac{2x^2}{1+2x}\right).
\end{equation*}Then $f(x)>0$.
\end{lemma}

\begin{proof}
    Differentiation yields
    \begin{align}\label{lm2.4-1}
        f'(x)&=2\left[\psi(x+1)-\psi(2x+1)+\frac{2x(1+x)}{(1+2x)(1+2x-2x^2)}\right]\triangleq2\hat{f}(x),\\ \nonumber
        \hat{f}'(x)&=\psi'(x+1)-2\psi'(2x+1)+\frac{2(1+2x+2x^2+8x^3+4x^4)}{(1+2x)^2(1+2x-2x^2)^2},
    \end{align}
which together with \eqref{lm2.2-1} implies
\begin{align*}
   \hat{f}'(x)&>\mathcal{L}_x(x,4/5)-2\mathcal{L}_x(2x,2/5)+\frac{2(1+2x+2x^2+8x^3+4x^4)}{(1+2x)^2(1+2x-2x^2)^2}  \\
   &=\dfrac{\left(\begin{array}{l}
    5533 + 37994 x + 92054 x^2 + 935456 x^3 + 11448028 x^4+ 67479864 x^5\\ 
    \quad+232646232 x^6 + 513934848 x^7 + 747879552 x^8+717361920 x^9\\
     \qquad+ 442143360 x^{10} + 158630400 x^{11} + 18662400 x^{12}
   \end{array}\right)}{2\left[\begin{array}{l}
    (1+2x)^2 (1+2x-2x^2)^2(17+15x+15x^2)\\
    \quad\times(11+36x+36x^2)(11+30x+60x^2)(5+36x+72x^2)
   \end{array}\right]}>0.
\end{align*}
That is to say, $\hat{f}(x)$ is strictly increasing on $(0,\infty)$. According to this with \eqref{lm2.4-1} and $\hat{f}(0)=f(0)=0$, it follows that $f(x)>0$ for $0<x<(\sqrt{3}+1)/2$.
\end{proof}

The following lemma offers a simple criterion to determine the sign of a class of special polynomial, which is called a Positive-Negative type (PN type for short) polynomial (see \cite{YTW2020,YT2018,YQCZ2017}) while its opposite is called a Negative-Positive type polynomial. More information related to these polynomials and the series version can be found in \cite{YQCZ2018,Yang2018,YangT2019}.

\begin{lemma}({\rm \cite[Lemma 4]{YT2018}})\label{lm-4}
Let $P_n(x)$ be a Positive-Negative type polynomial of degree $n$, which is given by
\begin{equation*}
P_n(x)=\sum_{k=0}^ma_kx^k-\sum_{k=m+1}^na_kx^k,
\end{equation*} where $a_k\geq0$ for all $k\geq0$ and there exist at least two integers $0\leq k_1\leq m$ and $m+1\leq k_2\leq n$ such that $a_{k_1},a_{k_2}\neq0$.
Then there exist $x_0\in(0,\infty)$ such that
$P_n(x_0)=0$ and $P_n(x)>0$ for $x\in(0,x_0)$ and $P_n(x)<0$ for $x\in(x_0,\infty).$  
\end{lemma}
\begin{remark}\label{rmk-1}
For a PN type polynomial, it can be easily seen that $P_n(x)>0$ for $x\in(0,x_1)$ if $P_n(x_1)>0$ and $P_n(x)<0$ for $x\in(x_2,\infty)$ if $P_n(x_2)<0$. Correspondingly, for a NP type polynomial $P_n(x)$, the signs of $P_n(x)$ on $(0,x_1)$ and $(x_2,\infty)$ are reversed. 
\end{remark}

By Applying Lemma \ref{lm-4},  we list some examples as PN or NP type polynomials below, which are used to prove the main theorem.
\begin{example} \label{example-1}
Let us introduce several PN type polynomials as
\begin{align*}
  p_0(x)&=1802+51043x-183763x^2-789936x^3-613605x^4,\\
  p_1(x)&=\left[\begin{array}{l}
       158486648 + 11733123789 x + 31469130662 x^2 + 40356637167 x^3\\
  \quad+19492784598 x^4 - 25810079475 x^5 - 68587238430 x^6 - 84246205050 x^7\\
  \quad- 76169872800 x^8 - 54370440000 x^9 - 28066500000 x^{10} - 9112500000
   x^{11}
  \end{array}\right],\\
  p_2(x)&=\left[\begin{array}{l}
9221+85665x+368643x^2+918717x^3+1422360 x^4\\
\quad+1354842x^5+652860x^6-76140x^7-291600x^8-145800 x^9 
  \end{array}\right],\\
  p_3(x)&=\left[\begin{array}{l} 3699375712 + 26523455964 x + 60608029362 x^2+ 68646948321 x^3+ 32033337786 x^4\\
     \quad- 25862152266 x^5-50136084000x^6-35657677440 x^7-11393978400x^8\end{array}\right],\\
   p_4(x)&=37+90x-93x^2-636x^3-810x^4-540x^5.
\end{align*}
Since $p_0(3/20)=75107551/32000$, $p_1(1/5)=64124455182553/15625$, $p_2(1)=4298768$, $p_3(1)=68461255039$ and $p_4(9/25)=21101408/1953125$, we conclude from Lemma \ref{lm-4} that
\begin{align}\label{exp1-1}
p_0(x)&>0 \quad\text{for}\quad x\in[0,3/20],\\ \label{exp1-2}
p_1(x)&>0 \quad\text{for}\quad x\in[0,1/5],\\ \label{exp1-3} 
p_2(x)&>0\quad\text{for}\quad x\in[0,1],\\ \label{exp1-4}
p_3(x)&>0 \quad\text{for}\quad x\in[0,1],\\ \label{exp1-5}
p_4(x)&>0 \quad\text{for}\quad x\in[0,9/25].
\end{align}
\end{example}

\begin{example}\label{example-2}
In the following, we will use several NP type polynomials, which are given by 
  \begin{align*}
q_0(x)&=-11 - 5 x + 11 x^2 + 5 x^3,\hspace{1.12cm} q_1(x)=-5+129x+131x^2+33x^3,\\
q_2(x)&=-65+254x+242x^2+49x^3,\quad q_3(x)=-84+222x+157x^2+19x^3,\\ q_4(x)&=-45+101x+43x^2+2x^3,\hspace{0.8cm}  q_5(x)=-11+23x+4x^2.
\end{align*}
Since $q_0(1/2)=-81/8$, $q_1(1/2)=771/8$, $q_2(1/2)=1029/8$, $q_3(1/2)=549/8$, $q_4(1/2)=33/2$ and $q_5(1/2)=3/2$, it follows from Lemma \ref{lm-4} that 
$q_0(x)<0$ for $x\in(0,1/2)$ and each of $q_j(x)$ ($1\leq j\leq5$) has a unique root on $(0,1/2)$. Suppose that $x_j$ is the positive root of $q_j(x)$ for $1\leq j\leq5$. A simple numerical computation gives $x_1=0.03733\cdots<$ $x_2=0.2114\cdots<x_3=0.3085\cdots<x_4=0.3822\cdots<x_5=0.4439\cdots$. In other words, by Remark \ref{rmk-1}, we see that $q_{j}(x)<0$ implies $q_{j+1}(x)<0$ on $(0,1/2]$ for $1\leq j\leq4$.
\end{example}

\section{Proof of Theorem \ref{thm-1}}

In this section, we give the proof of Theorem \ref{thm-1} in the case of $0<x+y\leq1$. Without loss of generality, we assume that $x\leq y$. 
\begin{proofa}
Let $0<x\leq y\leq1-x$ (namely, $0<x\leq1/2$) and 
\begin{align}\nonumber
F(x,y)&=\log\left[\frac{xyB(x,y)}{x+y}\right]-\log\left(1-\frac{2xy}{x+y+1}\right)\\ \label{pr-1}
&=\log\left[\frac{\Gamma(x+1)\Gamma(y+1)}{\Gamma(x+y+1)}\right]-\log\left(1-\frac{2xy}{x+y+1}\right).
\end{align}

Differentiating $F$ with respect to $x$ (res. $y$) gives
\begin{equation}\label{pr-2}
     \frac{\partial F}{\partial x}=\psi(x+1)-\psi(x+y+1)+\frac{2y(1+y)}{(1+x+y)(1+x+y-2xy)}
\end{equation}
   and
\begin{align}\nonumber
    \frac{\partial F}{\partial y}&=\psi(y+1)-\psi(x+y+1)+\frac{2x(1+x)}{(1+x+y)(1+x+y-2xy)}\\ \label{pr-3}
    &=\psi(y+1)-\psi(y+x)-\frac{1}{x+y}+\frac{2x(1+x)}{(1+x+y)(1+x+y-2xy)}.
\end{align}

We divide into two cases to complete the proof.
\begin{description}[leftmargin=1em]
\item[\sl Case 1:] $1/5\leq x\leq1/2$ and $x\leq y\leq 1-x$.\\
By taking $n=3$ into Lemma \ref{lm-1}, it follows from \eqref{pr-3} that
\begin{align*}
    \frac{\partial F}{\partial y}&>(1-x)\left[\frac{1}{x+y+3}+\sum_{i=0}^{2}\frac{1}{(y+i+1)(y+i+x)}\right]\\
    &\hspace{2cm}-\frac{1}{x+y}+\frac{2x(1+x)}{(1+x+y)(1+x+y-2xy)}\\
    &=\frac{xQ(x,y)}{(1+x+y-2xy)\prod_{j=1}^{3}\big[(y+j)(x+y+j)\big]},
\end{align*}
where $Q(x,y)=-q_0(x)+\sum_{k=1}^5q_k(x)y^k-(1-2x)y^6$
and $q_k(x)$ ($0\leq k\leq5$) is given in Example \ref{example-2}. From Example \ref{example-2}, we see that $Q(x,y)$ is a PN type polynomial of $y$ for $x\in(0,1/2]$. So we conclude, by Lemma \ref{lm-4}, that
$Q(x,y)>0$ for $1/5<x\leq1/2$ and $x\leq y\leq 1-x$ together with
\begin{equation*}
Q(x,1-x)=\frac{4}{625}(1-x)\Big[252+(5x-1)\big(7137+(1-x)(24365+375x^2)+5300x^2\big)\Big]>0.
\end{equation*}
Hence, $F(x,y)\geq F(x,x)=f(x)>0$ for $1/5\leq x\leq1/2$ and $x\leq y\leq 1-x$ by Lemma \ref{lm-3}.
\item[\sl Case 2:] $0<x<1/5$ and $x\leq y\leq 1-x$.\\
In this case, we denote by $D$ and $\partial D$ the trapezoidal domain $\{(x,y)|\ x<y<1-x,\ 0<x<1/5\}$ and its boundary. Then we will prove that $F(x,y)$ has no extreme value for $(x,y)\in D$.

By \eqref{pr-2} and \eqref{pr-3}, 
\begin{equation}\label{pr-4}
    G(x,y)=\frac{\partial F(x,y)}{\partial x}-\frac{\partial F(x,y)}{\partial y}=\psi(x+1)-\psi(y+1)-\frac{2(x-y)}{1+x+y-2xy}.
\end{equation}
We continue to divide into three subcases.
\begin{description}[leftmargin=1em]
\item[\it Subcase A:] $y\geq x+9/25$.\\
In this case, by differentiation, we obtain
\begin{align}\nonumber
     \frac{\partial G}{\partial x}&=\psi'(x+1)-\frac{2(1+2y-2y^2)}{(1+x+y-2xy)^2},\\ \label{pr-5}
     \frac{\partial^2 G}{\partial x\partial y}&=\frac{\partial^2 G}{\partial y\partial x}=\frac{12(y-x)}{(1+x+y-2xy)^3}>0,
\end{align}
which gives
\begin{equation}\label{pr-6}
    \frac{\partial G(x,y)}{\partial x}>\frac{\partial G(x,x+9/25)}{\partial x}=\psi'(x+1)-\frac{913+350x-1250x^2}{2(17+16x-25x^2)^2}\triangleq g(x).
\end{equation}
It follows from \eqref{lm2.2-1} and \eqref{exp1-1} that
\begin{align*}
 g(x)&>\mathcal{L}_x(x,4/5)-\frac{913+350x-1250x^2}{2(17+16x-25x^2)^2}\\
 &=\frac{p_0(x)+307230 x^5 + 823500 x^6 + 675000 x^7}{2 (17 + 15 x + 15 x^2) (17+16x-25x^2)^2 (11 + 36 x + 36 x^2)}>0
\end{align*}for $x\in(0,3/20]$. On the other hand, by \eqref{lm2.2-2} and \eqref{exp1-2},
\begin{align*}
g'(x)&=\psi''(x+1)+\frac{11633-21600x-13125x^2+31250x^3}{(17+16x-25x^2)^3}\\
&<\mathcal{L}_{xx}(x,4/5)+\frac{11633-21600x-13125x^2+31250x^3}{(17+16x-25x^2)^3}\\
&=-\frac{127679911(10x-1)+xp_1(x)}{2(17+15x+15x^2)^2(17+16x-25x^2)^3(11+36x+36x^2)^2}<0
\end{align*}for $x\in(1/10,1/5)$, which implies that $g(x)>g(1/5)=0.001914\cdots>0$ for $x\in(1/10,1/5)$. In summary, $\partial G/\partial x>0$ for $0<x<1/5$ by \eqref{pr-6}.
Further, by \eqref{lm2.2-2}, \eqref{exp1-3} and \eqref{pr-4}, for $y\in[0,1]$,
\begin{align*}
    \frac{d^2}{dy^2}G(0,y)&=-\frac{4}{(1+y)^3}-\psi''(1+y)<-\frac{4}{(1+y)^3}-\mathcal{L}_{yy}(y,2/5)\\
    &=-\frac{p_2(y)}{2(1+y)^3(11+15y+15y^2)^2(5+18y+18y^2)^2}<0,
\end{align*}which shows that $G(0,y)$ is strictly concave on $[0,1]$. 
Hence, 
\begin{equation*}
    G(x,y)>G(0,y)\geq\min\big\{G(0,0),G(0,1)\big\}=0
\end{equation*}
for $0<x<1/5$ and $x+9/25\leq y<1$.
\item[\it Subcase B:] $9/25<y<x+9/25$.\\
In this case, $y-9/25<x<1/5$. It is easy to see from \eqref{pr-5} that $\partial G/\partial y$ is strictly increasing for $0<x<y$ and so by \eqref{lm2.2-1}
\begin{align*}
\frac{\partial G(x,y)}{\partial y}&>\frac{\partial G(y-9/25,y)}{\partial y}\\
&=\frac{13 + 2150 y - 1250 y^2}{2 (8 + 34 y - 25 y^2)^2} -\psi'(y+1)>\frac{13 + 2150 y - 1250 y^2}{2 (8 + 34 y - 25 y^2)^2} -\mathcal{L}_y(y,2/5)\\
&=\dfrac{5275352+(25y-9)\left[
\begin{array}{l}
     4404553 + 18643550 y + 55576875 y^2+88996875 y^3 \\
    \quad + 9375000 y^4+843750y^4(1-y)(57+50y)
\end{array}
\right]}{6250(11 + 15 y + 15 y^2)(5+18y+18 y^2) (8+34y-25 y^2)^2}>0.
\end{align*}
Moreover, by \eqref{lm2.2-2} and \eqref{exp1-4},
\begin{align*}
    \frac{d^2}{dx^2}G(x,9/25)&=\psi''(x+1)+\frac{25564}{(34+7x)^3}<\mathcal{L}_{xx}(x,4/5)+\frac{25564}{(34+7x)^3}\\
    &=-\frac{p_3(x)+200037600 x^9}{2(34+7x)^3(17+15x+15x^2)^2(11+36x+36x^2)^2}<0
\end{align*}
for $x\in(0,1/5)$. That is to say, $G(x,9/25)$ is strictly concave on $(0,1/5)$ and thereby
\begin{align*}
  G(x,y)>G(x,9/25)&>\min\big\{G(0,9/25),G(1/5,9/25)\big\}\\
  &=\min\{0.0554\cdots,0.04015\cdots\}>0. 
\end{align*}
\item[\it Subcase C:] $x<y\leq9/25$.\\
As mentioned in Subcase B, $\partial G/\partial y$ is strictly increasing for $0<x<y$ and so by \eqref{lm2.2-1}
\begin{align*}
 \frac{\partial G(x,y)}{\partial y}&>\frac{\partial G(0,y)}{\partial y}=\frac{2}{(y+1)^2}-\psi'(y+1)>\frac{2}{(y+1)^2}-\mathcal{L}_{y}(y,2/5)\\
 &=\frac{p_4(y)}{2 (1 + y)^2 (11 + 15 y + 15 y^2) (5 + 18 y + 18 y^2)}>0
\end{align*}for $0<x<y\leq9/25$, where $p_4(y)>0$ for $y\in(0,9/25]$ follows from \eqref{exp1-5}. Hence,
$G(x,y)>G(x,x)=0$ by \eqref{pr-4}.
\end{description}
In conclusion, $G(x,y)>0$ for all $(x,y)\in D$, which shows that $F(x,y)$ has no extreme value for $(x,y)\in D$. Otherwise there exists a point $(x_0,y_0)\in D$ such that 
\begin{equation*}
    \frac{\partial F(x_0,y_0)}{\partial x}=\frac{\partial F(x_0,y_0)}{\partial y}=0\quad\Longleftrightarrow\quad G(x_0,y_0)=0.
\end{equation*}
For $(x,y)\in\partial D$, $F(x,y)$ has the following properties: 
\begin{enumerate}[itemindent=*,label=(\roman*)]
    \item when $x+y=1$, $F(x,y)>0$ by \eqref{sec3-1} and \eqref{pr-1};
    \item when $x=0$, it is easy to see from \eqref{pr-1} that $F(0,y)=0$;
    \item when $x=y$, by Lemma \ref{lm-3}, $F(x,x)=f(x)>0$ for $x\in(0,1)$;
    \item when $x=1/5$, it has been proved in Case 1 that $F(1/5,y)>0$ for $1/5\leq x\leq4/5$.
\end{enumerate}
Due to the continuity of $F(x,y)$ on $\Bar{D}$, it follows that $F(x,y)$ can only attain its minimal value at the boundary of $\bar{D}$, that is to say,
\begin{equation*}
    F(x,y)\geq\min_{(x,y)\in\partial D}\big\{F(x,1-x),F(0,y),F(x,x),F(1/5,y)\big\}=0
\end{equation*}with the equality only for $x=0$.
\end{description}
This completes the proof. \qed
\end{proofa}

\noindent{\bf Funding}\\
\begin{small}
This work was supported by the National Natural Science Foundation of China (11971142) and the Natural Science Foundation of Zhejiang Province (LY19A010012).    
\end{small}

\bigskip

\noindent{\bf Availability of data and materials}\ \   
\begin{small}
    Not applicable.
\end{small}
\bigskip

\noindent{\bf Declarations}

\noindent{\bf Conflict of interest}\  \ 
\begin{small}
    The authors declare that they have no conflict of interest.
\end{small}

\end{document}